\documentclass[12pt,a4paper]{amsart}

\usepackage{amsrefs}
\usepackage[all]{xy}

\newtheorem{Thm}{Theorem}[section]
\newtheorem{Cor}[Thm]{Corollary}
\newtheorem{Lem}[Thm]{Lemma}
\newtheorem{Prop}[Thm]{Proposition}

\newtheorem{step}{Step}[section]
\renewcommand{\thestep}{\Roman{step}}

\newcommand{\vep}{\varepsilon}

\newcommand{\cM}{\mathcal{M}}
\newcommand{\cO}{\mathcal{O}}
\newcommand{\cP}{\mathcal{P}}
\newcommand{\cS}{\mathcal{S}}
\newcommand{\cU}{\mathcal{U}}
\newcommand{\cX}{\mathcal{X}}
\newcommand{\cY}{\mathcal{Y}}

\newcommand{\bp}{\mathbf{p}}
\newcommand{\bq}{\mathbf{q}}
\newcommand{\br}{\mathbf{r}}
\newcommand{\bs}{\mathbf{s}}

\newcommand{\bbZ}{\mathbb{Z}}

\let\mod=\undefined
\DeclareMathOperator{\GL}{GL}
\DeclareMathOperator{\GP}{GP}
\DeclareMathOperator{\Id}{Id}
\DeclareMathOperator{\Aut}{Aut}
\DeclareMathOperator{\End}{End}
\DeclareMathOperator{\ext}{ext}
\DeclareMathOperator{\Ext}{Ext}
\DeclareMathOperator{\Hom}{Hom}
\DeclareMathOperator{\mod}{mod}
\DeclareMathOperator{\idim}{inj.dim}
\DeclareMathOperator{\pdim}{proj.dim}

\newcommand{\ov}{\overline}

\begin{document}

\title[Algebras with irreducible module varieties III]{Algebras with irreducible module varieties III: Birkhoff varieties}

\author{Grzegorz Bobi\'nski}
\address{Grzegorz Bobi\'nski\newline
Faculty of Mathematics and Computer Science\newline
Nicolaus Copernicus University\newline
ul. Chopina 12/18\newline
87-100 Toru\'n\newline
Poland}
\email{gregbob@mat.umk.pl}

\begin{abstract}
We study a family of affine varieties arising from a version of an old problem due to Birkhoff asking for the classification of embeddings of finite abelian $p$-groups. We show that all of these varieties are irreducible and have a dense orbit.
\end{abstract}

\maketitle


\section{Introduction and main result}


\subsection{Introduction}
Let $\Lambda$ be a ring. The \emph{submodule category} $\cS (\Lambda)$ of $\Lambda$-modules has as objects pairs $(U, M)$, where $M$ is a finitely generated $\Lambda$-module and $U \subseteq M$ is a submodule of $U$. A morphism $f \colon (U, M) \to (V, N)$ in $\cS (\Lambda)$ is given by a $\Lambda$-module homomorphism $f \colon M \to N$ such that $f (U) \subseteq V$. Even in cases when the category $\mod \Lambda$ of finitely generated $\Lambda$-modules is well understood, it can be surprisingly difficult to describe $\cS (\Lambda)$.

For $\Lambda = \bbZ / (p^n)$ the problem of classifying all objects in $\cS (\Lambda)$ has been mentioned already by Birkhoff~\cite{Birkhoff}. In this case, $\cS(\Lambda)$ is a Krull--Remak--Schmidt category, i.e.\ all objects are direct sums of indecomposable objects, and these are uniquely determined up to isomorphism. There are $n$ indecomposable $\Lambda$-modules in $\mod (\Lambda)$, namely $\bbZ / (p^i)$ with $1 \le i \le n$. All of these are uniserial. Richman and Walker~\cite{Richman} have proved that the categories $\cS (\Lambda)$ contain only finitely many indecomposable objects for $n \le 5$ and infinitely many, otherwise.

Moving from uniserial rings to finite-dimensional algebras, Ringel and Schmidmeier~\cite{Ringel} have thoroughly studied $\cS (\Lambda)$ for $\Lambda = K[X]/(X^n)$. Again, $\cS (\Lambda)$ is well understood for $n \le 5$, and they have also described the first difficult case $n = 6$. For $n > 6$, the category $\cS(\Lambda)$ is of \emph{wild representation type}. This means that a classification of objects in $\cS(\Lambda)$ is impossible in a mathematically precise sense. We refer also to~\cite{KLM} for a link between the problem and singularity theory.

In this article we also concentrate on the case $\Lambda = K[X]/(X^n)$. But instead of studying $\cS(\Lambda)$, we study a closely related category $\mod A$ of finite-dimensional modules over a finite-dimensional $K$-algebra $A$. The algebra $A$ turns out to be a $1$-Iwanaga--Gorenstein algebra such that $\cS (\Lambda)$ is equivalent to the subcategory $\GP (A) \subseteq \mod A $ of Gorenstein-projective $A$-modules.

We will not attempt to classify all $A$-modules, but as our main result we show that for each dimension vector, there is \emph{generically} only one $A$-module. In particular, the algebra $A$ has the dense orbit property in the sense of Chindris, Kinser and Weyman~\cite{CKW}.

\subsection{Main result}
Throughout, we work over an algebraically closed field $K$. For $m \geq 1$ and $d_0, d_1 \geq 0$, let $\cX_m (d_0, d_1)$ be the affine variety consisting of the triples $M = (M_0, M_1, h_M)$ such that $M_0$ is a $(d_0 \times d_0)$-matrix, $M_1$ is a $(d_1 \times d_1)$-matrix, and $h_M$ is a $(d_0 \times d_1)$-matrix, such that $M_0^m = 0$, $M_1^m = 0$, and $M_0 h_M = h_M M_1$. By convention, if $d', d'' \geq 0$ and either $d' = 0$ or $d'' = 0$, then there is a unique $(d' \times d'')$-matrix, which behaves like zero with respect to multiplication. The following is the main result of the paper.

\begin{Thm}\label{thm:intro4}
The variety $\cX_m (d_0, d_1)$ is irreducible for all $m \geq 0$ and $d_0, d_1 \geq 0$.
\end{Thm}

The group $G_{d_0,d_1} := \GL_{d_0}(K) \times \GL_{d_1}(K)$ acts on $\cX_m(d_0,d_1)$ by
\[
(g_0, g_1) \cdot (M_0, M_1, h_M) := (g_0 M_0 g_0^{-1}, g_1 M_1 g_1^{-1},g_0 h_M g_1^{-1}).
\]
We prove Theorem~\ref{thm:intro4} by showing that $\cX_m (d_0, d_1)$ has a unique dense $G_{d_0,d_1}$-orbit.

We explain now a connection of the above theorem with the theory of geometrically irreducible algebras developed in~\cites{BobinskiSchroer2017a, BobinskiSchroer2017b}. Recall that an algebra $A$ is called \emph{geometrically irreducible} if and only if, for each dimension $d$, the connected components of the variety $\mod (A, d)$ of $d$-dimensional $A$-module structures are irreducible. We formulated in~\cite{BobinskiSchroer2017a} a conjecture, which states that up to a trivial glueing procedure every geometrically irreducible algebra has (up to isomorphism) at most two simple modules. This conjecture has been verified for a wide class of algebras in~\cite{BobinskiSchroer2017b}*{Theorem~1.2}. Geometrically irreducible algebras with exactly one simple module are precisely the local algebras~\cite{BobinskiSchroer2017a}*{Proposition~1.5}. Now let $1 \leq n \leq m$ and put
\[
A (m, n) :=
\begin{bmatrix}
\Lambda & \Gamma \\ 0 & \Lambda
\end{bmatrix},
\]
where $\Lambda := K [X] / (X^m)$ and $\Gamma := (\Lambda \otimes_K \Lambda) / (\sum_{i = 0}^n X^i \otimes X^{n - i})$ (which we treat as a $\Lambda$-bimodule). It has been proved in~\cite{BobinskiSchroer2017b}*{Theorem~1.1} that if $A$ is a geometrically irreducible algebra with exactly two simple modules, then up to Morita equivalence and the glueing procedure mentioned above either $A \cong A (m, 1)$, for some $m \geq 1$, or $A \cong A (m, m - 1)$, for some $m \geq 2$. It was left open in~\cite{BobinskiSchroer2017b} if the algebras $A (m, 1)$ and $A (m, m - 1)$ are actually geometrically irreducible, for $m \geq 2$. The following reformulation of Theorem~\ref{thm:intro4} answers this question in the case of the former family.

\begin{Cor} \label{cor main}
The algebra $A (m, 1)$ is geometrically irreducible for all $m \geq 1$.
\end{Cor}

It remains open if $A (m, m - 1)$ is geometrically irreducible for $m \geq 3$.

\subsection{Notation}
Throughout the paper $m \geq 1$ is a fixed integer and $\Lambda := K [X] / (X^m)$. Moreover $A := A (m, 1)$. Note that $A$ is isomorphic to the matrix algebra $\left[
\begin{smallmatrix}
\Lambda & \Lambda \\ 0 & \Lambda
\end{smallmatrix}
\right]$.

\subsection{Acknowledgments}
The author acknowledges the support of the National Science Center grant no.\ 2015/17/B/ST1/01731. He would also like to thank the University of Bonn for its hospitality during his visits in July 2016 and July 2017. In particular, he thanks Jan Schro\"er for many discussions. The paper is a continuation of a joint project on geometrically irreducible algebras and Jan Schro\"er has influenced the paper a lot.


\section{Preliminary observations}


\subsection{Interpretation of the variety}
If $d \geq 0$ and $U$ is a $(d \times d)$-matrix such that $U^m = 0$, then $U$ induces a $\Lambda$-module structure on $K^d$. By abuse of
notation we denote this $\Lambda$-module also by $U$. Consequently, if $d_0, d_1 \geq 0$, $M_0$ is a $(d_0 \times d_0)$-matrix, $M_1$ is a $(d_1 \times d_1)$-matrix, and $h_M$ is $(d_0 \times d_1)$-matrix, then $M = (M_0, M_1, h_M) \in \cX_m (d_0, d_1)$ if and only if $M_0^m = 0$, $M_1^m = 0$, and $h_M \in \Hom_\Lambda (M_1, M_0)$, where we identify $h_M$ with the induced map $K^{d_1} \to K^{d_0}$.

Taking the above into account we denote by $\cM$ the category consisting of the triples $M = (M_0, M_1, h_M)$ such that $M_0$ and $M_1$ are finite-dimensional $\Lambda$-modules and $h_M \in \Hom_\Lambda (M_1, M_0)$. If $M$ and $N$ are objects of $\cM$, then $\Hom_{\cM} (M, N)$ consists of the pairs $f = (f (0), f (1))$ of $\Lambda$-module homomorphisms $f (0) \colon M_0 \to N_0$ and $f (1) \colon M_1 \to N_1$, such that $f (0) h_M = h_N f (1)$. One easily sees that $\cM$ is an abelian category. In fact, $\cM$ is equivalent to the category $\mod(A)$ of finite-dimensional left $A$-modules. Throughout the rest of the paper we treat this equivalence as an identification. In particular, if $M$ and $N$ are objects of $\cM$, then we write $\Hom_A (M, N)$ instead of $\Hom_{\cM} (M, N)$, etc.

Note that according to~\cite{BobinskiSchroer2017a}*{Proposition~4.2} Corollary~\ref{cor main} implies that $A$ is an Iwanaga--Gorenstein algebra, i.e.\ $\pdim_A D (A) < \infty$ and $\idim_A A < \infty$. In fact, \cite{GeissLeclercSchroer}*{Theorem~1.2} states that $A$ is even $1$-Iwanaga--Gorenstein, i.e.\ $\pdim_A D (A) = 1 = \idim_A A$. Moreover, \cite{GeissLeclercSchroer}*{Theorem~10.9} (see also~\cite{LuoZhang}*{Corollary 6.1}) says that a triple $M$ corresponds to a Gorenstein-projective $A$-module if and only if $h_M$ is injective, thus the category $\GP (A)$ of Gorenstein-projective $A$-modules is equivalent to the submodule category $\cS (\Lambda)$.

\subsection{Representation theory and geometry of the truncated polynomial algebras} \label{sub truncated} \label{sub geometry}
We need to recall some facts about the representation theory and geometry of $\Lambda$-modules.

It is well-known that, for a given $d \geq 0$, the isomorphism classes of $d$-dimensional $\Lambda$-modules are parameterized by the partitions of $d$ with parts at most $m$ (we denote the set of such partitions by $\cP_m (d)$ and from now on we assume that all partitions have parts at most $m$).
A~partition $\bp = (p_1, \ldots, p_l) \in \cP_m (d)$ corresponds to a $\Lambda$-module $U_\bp$ with the action of $X$ given by the matrix
\[
J_\bp := \begin{bmatrix}
J_{p_1} & 0 & \cdots & 0
\\
0 & J_{p_2} & \ddots & \vdots
\\
\vdots & \ddots & \ddots & 0
\\
0 & \cdots & 0 & J_{p_l}
\end{bmatrix},
\]
where $J_p$ denotes the nilpotent Jordan matrix of size $p$ (with $1$'s below the diagonal). Obviously, in the situation above we have
\[
U_\bp \cong U_{(p_1)} \oplus \cdots \oplus U_{(p_l)}.
\]
A partition $\bp \in \cP_m (d)$ is called \emph{maximal} if either $\bp$ is the empty partition (if $d = 0$) or $\bp = (m, \ldots, m, r)$, for some $1 \leq r \leq m$ (if $d > 0$).

Now we describe homomorphisms between $\Lambda$-modules. It is sufficient to describe homomorphisms between the indecomposable $\Lambda$-modules, i.e.\ the modules of the form $U_{(p)}$, $1 \leq p \leq m$.
First observe, that if $1 \leq p \leq m$, then $U_{(p)} \cong K [X] / (X^p)$, where the isomorphism sends the $i$-th standard basis vector to $X^{i - 1}$. We will treat this isomorphism as an identification. With this identification in mind, if $1 \leq p, q \leq m$, then every homomorphism $U_{(q)} \to U_{(p)}$ is given by the multiplication by a polynomial $\Phi \in K [X]$, such that $X^{p - q} \mid \Phi$ (this condition is empty if $q \geq p$). Moreover, $\Phi$ induces the zero map if and only if $X^p \mid \Phi$. In particular $\dim_K \Hom_\Lambda (U_{(q)}, U_{(p)}) = \min \{ p, q \}$.

In the rest of the paper we represent the homomorphisms between indecomposable $\Lambda$-modules by polynomials. It is worth noting that the polynomial $1$ represents a morphism $U_{(q)} \to U_{(p)}$ if and only if $q \geq p$, however it is the identity morphism if and only if $q = p$ (if $q > p$, then the morphism represented by $1$ is epi, but not mono). Observe that $\Psi$ represents an epimorphism $U_{(q)} \to U_{(p)}$ if and only if $\Psi$ has a non-zero constant term (in particular, $q \geq p$). Dually, $\Psi$ represents a monomorphism $U_{(q)} \to U_{(p)}$ if and only if $\frac{\Psi}{X^{p - q}}$ is a polynomial with a non-zero constant term (in particular, $q \leq p$).

If homomorphisms $\varphi \colon U_{(q)} \to U_{(p)}$ and $\psi \colon U_{(r)} \to U_{(q)}$ are represented by polynomials $\Phi$ and $\Psi$, respectively, then $\varphi \circ \psi$ is represented by the polynomial $\Phi \cdot \Psi$ (the usual multiplication of polynomials). We will sometimes write $\Phi \circ \Psi$ instead of $\Phi \cdot \Psi$ in order to stress this fact.

We note the following factorization property for homomorphisms of $\Lambda$-modules (which can be proved using the above).
Let $r \leq q \leq p$. If either $\varphi \colon U_{(q)} \to U_{(p)}$ or $\varphi \colon U_{(r)} \to U_{(q)}$, and $\varphi$ is mono, then every homomorphism $U_{(r)} \to U_{(p)}$ factors through $\varphi$. We have the dual statement for epimorphisms.

It is well known that, for each $p$, $U_{(p)} \cong \Hom_K (U_{(p)}, K)$ (as $\Lambda$-modules), and an isomorphism can be given by a map which sends $1 \in U_{(p)}$ to the map $X^i \mapsto \delta_{i, p - 1}$, $i = 0, \ldots, p - 1$, where $\delta_{x, y}$ is the Kronecker delta. Then, if $\Phi \in K [X]$ represents a homomorphism $\varphi \colon U_{(q)} \to U_{(p)}$, then $X^{q - p} \cdot \Phi$ represents $\Hom_K (\varphi, K) \colon U_{(p)} \to U_{(q)}$.

For $d \geq 0$ let $\cY_m (d)$ be the set of $(d \times d)$-matrices $U$ such that $U^m = 0$. Recall that each $U \in \cY_m (d)$ can viewed as a $\Lambda$-module of dimension $d$. For a partition $\bp \in \cP_m (d)$ we denote by $\cO_\bp$ the set of $U \in \cY_m (d)$ such that $U \cong U_\bp$. This is an irreducible locally closed subset of $\cY_m (d)$ (note that $\cO_\bp$ is the orbit of $J_\bp$ with respect to the conjugation action of $\GL_d (K)$). It is classical (see for example~\cite{Gerstenhaber}) that $\cO_\bp \subseteq \ov{\cO_\bq}$ if and only if $\bp \leq \bq$ in the dominance order, i.e.\ $p_1 + \cdots + p_i \leq q_1 + \ldots + q_i$, for each $i$.

\subsection{Weakly indecomposable partitions} \label{sub weakly}
We call a pair $(\bp, \bq)$ of partitions (with parts at most $m$) \emph{weakly indecomposable} if either
\[
p_1 \geq q_1 \geq p_2 \geq q_2 \geq \cdots \qquad \text{or} \qquad q_1 \geq p_1 \geq q_2 \geq p_2 \geq \cdots.
\]
Here and later we always use notation
\[
\bp = (p_1, p_2, \ldots) \qquad \text{and} \qquad \bq = (q_1, q_2, \ldots),
\]
where we extend partitions by zeros. In the former case we say that $(\bp, \bq)$ is of \emph{mono type}, while in the latter case we say that $(\bp, \bq)$ is of \emph{epi type}. Obiously, $(\bp, \bq)$ can be both of mono and epi type at the same time (this happens if $p_i = q_i$, for all $i$). We want to define an $A$-module $M_{(\bp, \bq)} := (U_\bp, U_\bq, h_{(\bp, \bq)})$ for each weakly indecomposable pair $(\bp, \bq)$ (we usually write $M_{\bp, \bq}$ and $h_{\bp, \bq}$ instead of $M_{(\bp, \bq)}$ and $h_{(\bp, \bq)}$, respectively). If $(\bp,\bq)$ is of mono type, we define the map
$h_{\bp, \bq}$ by the following matrix
\[
\begin{bmatrix}
X^{p_1 - q_1} & 0 & 0 & \cdots
\\
1 & X^{p_2 - q_2} & 0
\\
0 & 1 & X^{p_3 - q_3} & \ddots
\\
\vdots & \ddots &  \ddots & \ddots
\end{bmatrix},
\]
(with respect to the standard decompositions $U_{\bp} := \bigoplus_{i \geq 1} U_{(p_i)}$ and $U_{\bq} := \bigoplus_{i \geq 1} U_{(q_i)}$). In the epi case we define the map $h_{\bp, \bq}$ dually, i.e.\
\[
h_{\bp, \bq} :=
\begin{bmatrix}
1 & X^{p_1 - q_2} & 0 & \cdots
\\
0 & 1 & X^{p_2 - q_3} & \ddots
\\
0 & 0 & 1 & \ddots
\\
\vdots & & \ddots & \ddots
\end{bmatrix}
\]
Note that if the above cases overlap, namely, if $p_i = q_i$ for each $i$, then one may use both definitions. The definitions differ, however they give isomorphic modules, namely, the direct sum of $M_{(p_i), (q_i)}$, $i \geq 1$ (see Lemma~\ref{lemma iso}), hence by abuse of notation we denote them by the same symbol. It should always be clear from the context, which particular version of the definition we have in mind.

If $\bp$ and $\bq$ are partitions, then by $\bp \cup \bq$ we denote the partition with the parts
\[
p_1, q_1, p_2, q_2, \ldots
\]
(not necessarily in this order). Similarly, if $(\bp, \bq)$ and $(\br, \bs)$ are pairs of partitions, then $(\bp, \bq) \cup (\br, \bs) := (\bp \cup \br, \bq \cup \bs)$. Note that if $(\bp, \bq)$ is a weakly indecomposable pair of partitions and $0 < n \leq m$ is an integer, then the pair $(\bp, \bq) \cup ((n), (n))$ is weakly indecomposable of the same type as $(\bp, \bq)$. We have the following.

\begin{Lem} \label{lemma iso}
If $(\bp, \bq)$ is a weakly indecomposable pair of partitions and $0 < n \leq m$ is an integer, then
\[
M_{(\bp, \bq) \cup ((n), (n))} \cong M_{\bp, \bq} \oplus M_{(n), (n)}.
\]
\end{Lem}

\begin{proof}
In order to simplify the presentation we assume that $(\bp, \bq)$ is of mono type. Let $i_0$ be the minimal $i$ such that $n \geq p_i$. In particular, $i_0 := \ell (\bp) + 1$, if $n < p_{\ell (\bp)}$ (here and later $\ell (\bp)$ denotes the length of the partition $\bp$, i.e.\ the maximal $i$ such that $p_i > 0$; note that $\ell (\bp) = 0$, if $\bp$ is the empty partition). Then
\[
\bp \cup ((n)) = (p_1, \ldots, p_{i_0 - 1}, n, p_{i_0}, \ldots).
\]
If $i_0 = 1$ or $q_{i_0 - 1} \geq n$, then
\[
\bq \cup ((n)) = (q_1, \ldots, q_{i_0 - 1}, n, q_{i_0}, \ldots),
\]
otherwise
\[
\bq \cup ((n)) = (q_1, \ldots, q_{i_0 - 2}, n, q_{i_0 - 1}, \ldots).
\]
We concentrate on the former case, the latter can be treated similarly.

Let
\[
U' := \bigoplus_{i < i_0} U_{(p_i)} \qquad \text{and} \qquad U'' := \bigoplus_{i > i_0} U_{(p_i)}.
\]
Similarly,
\[
V' := \bigoplus_{j < i_0 - 1} U_{(q_j)} \qquad \text{and} \qquad V'' := \bigoplus_{j \geq i_0} U_{(q_j)}.
\]
Then
\begin{gather*}
U_\bp = U' \oplus U_{(p_{i_0})} \oplus U'' \qquad \text{and} \qquad U_\bq = V' \oplus U_{(q_{i_0 - 1})} \oplus V'',
\\
\intertext{while}
U_{\bp \cup (n)} = U' \oplus U_{(n)} \oplus U_{(p_{i_0})} \oplus U''
\\
\intertext{and}
U_{\bq \cup (n)} = V' \oplus U_{(q_{i_0 - 1})} \oplus U_{(n)} \oplus V'',
\end{gather*}
where $U_{(q_{i_0 - 1})}$ is the zero module, if $i_0 = 1$.

With respect to these decompositions
\begin{gather*}
h_{\bp, \bq} =
\begin{bmatrix}
\varphi_{1, 1} & \varphi_{1, 2} & 0
\\
0 & 1 & \varphi_{2, 3}
\\
0 & 0 & \varphi_{3, 3}
\end{bmatrix}
\\
\intertext{and}
h_{(\bp, \bq) \cup ((n), (n))} =
\begin{bmatrix}
\varphi_{1, 1} & \varphi_{1, 2} & 0 & 0
\\
0 & 1 & \Id_{U_{(n)}} & 0
\\
0 & 0 & 1 & \varphi_{2, 3}
\\
0 & 0 & 0 & \varphi_{3, 3}
\end{bmatrix},
\end{gather*}
(for some homomorphisms $\varphi_{1, 1} \colon V' \to U'$, $\varphi_{1, 2} \colon U_{(q_{i_0 - 1})} \to U'$, $\varphi_{2, 3} \colon V'' \to U_{(p_{i_0})}$, and $\varphi_{3, 3} \colon V'' \to U''$). Using appropriate row and column operations we transform the matrix $h_{(\bp, \bq) \cup ((n), (n))}$ to the form
\[
\begin{bmatrix}
\varphi_{1, 1} & \varphi_{1, 2} & 0 & 0
\\
0 & 1 & \varphi_{2, 3} & 0
\\
0 & 0 & \varphi_{3, 3} & 0
\\
0 & 0 & 0 & \Id_{U_{(n)}}
\end{bmatrix} =
\begin{bmatrix}
h_{\bp, \bq} & 0
\\
0 & h_{(n), (n)}
\end{bmatrix},
\]
and the claim follows.
\end{proof}

\subsection{Canonical decomposition} \label{subsect can}
A pair $(\bp, \bq)$ of partitions is called \emph{indecomposable}, if either $(\bp) = (n) = (\bq)$ for some $0 < n \leq m$, or
\begin{align*}
\ell (\bp) & = \ell (\bq) > 0 \quad \text{and} \quad p_1 > q_1 > p_2 > q_2 > \cdots > p_{\ell (\bp)} > q_{\ell (\bq)}, \text{ or}
\\
\ell (\bq) & = \ell (\bp) > 0 \quad \text{and} \quad q_1 > p_1 > q_2 > p_2 > \cdots > q_{\ell (\bq)} > p_{\ell (\bp)}, \text{ or}
\\
\ell (\bp) & = \ell (\bq) + 1 \quad \text{and} \quad p_1 > q_1 > p_2 > q_2 > \cdots > p_{\ell (\bp) - 1} > q_{\ell (\bq)} > p_{\ell (\bp)}, \text{ or}
\\
\ell (\bq) & = \ell (\bp) + 1 \quad \text{and} \quad q_1 > p_1 > q_2 > p_2 > \cdots > q_{\ell (\bq) - 1} > p_{\ell (\bp)} > q_{\ell (\bq)}.
\end{align*}
The following lemma explains the terminology.

\begin{Lem} \label{lemma ind}
Let $(\bp, \bq)$ be an indecomposable pair of partitions. Then $M_{\bp, \bq}$ is indecomposable.
\end{Lem}

\begin{proof}
Up to duality we may assume $(\bp, \bq)$ is of mono type. An endomorphism $f = (f (0), f(1))$ of $M_{\bp, \bq}$ is of the form
\begin{align*}
f (0) & =
\begin{bmatrix}
\lambda_1 + X a_{1 1} & X^{p_1 - p_2} a_{1 2} & X^{p_1 - p_3} a_{1 3} & \cdots
\\
a_{2 1} & \lambda_2 + X a_{2 2} & X^{p_2 - p_3} a_{2 3}
\\
a_{3 1} & a_{3 2} & \lambda_3 + X a_{3, 3} & \ddots
\\
\vdots & &  \ddots & \ddots
\end{bmatrix}
\\
\intertext{and}
f (1) & =
\begin{bmatrix}
\mu_1 + X b_{1 1} & X^{q_1 - q_2} b_{1 2} & X^{q_1 - q_3} b_{1 3} & \cdots
\\
b_{2 1} & \mu_2 + X b_{2 2} & X^{q_2 - q_3} b_{2 3}
\\
b_{3 1} & b_{3 2} & \mu_3 + X b_{3 3} & \ddots
\\
\vdots & &  \ddots & \ddots
\end{bmatrix},
\end{align*}
for some scalars $\lambda_i, \mu_i, \in K$ and polynomials $a_{i j}, b_{i j} \in K [X]$. By looking at the degree $p_i - q_i$ coefficients of the diagonal entries of the matrices $f (0) h_{\bp, \bq}$ and $h_{\bp, \bq} f (1)$ we get the equalities
\[
\lambda_1 = \mu_1 , \; \lambda_2 = \mu_2, \; \ldots.
\]
Similarly, the constant terms of the below diagonal entries of the matrices $f (0) h_{\bp, \bq}$ and $h_{\bp, \bq} f (1)$ give equalities
\[
\lambda_2 = \mu_1, \; \lambda_3 = \mu_1, \; \ldots.
\]
Consequently, $f = \lambda_1 \Id_{M_{\bp, \bq}} + f'$, where $f'$ is a radical morphism ($f'$ is a radical morphism in $\mod \Lambda$, hence $(f')^m = 0$), thus $\End_A (M_{\bp, \bq})$ is a local ring and the claim follows.
\end{proof}

Together with Lemma~\ref{lemma iso}, Lemma~\ref{lemma ind} implies that if $(\bp, \bq)$ is a weakly indecomposable pair of partitions, then $M_{\bp, \bq}$ is indecomposable if and only if the pair $(\bp, \bq)$ is indecomposable. In fact, using Proposition~\ref{prop generic} one may show that a pair $(\bp, \bq)$ of partitions is indecomposable if and only if the generic $A$-module $M$ with $M_0 \cong U_\bp$ and $M_1 \cong U_\bq$ is indecomposable.

For an arbitrary pair $(\bp, \bq)$ of partitions, we define now what we call the \emph{canonical decomposition} of $(\bp, \bq)$. We recommend the reader to study Example~\ref{example can} before reading the formal definition. First we choose subsets $I_0 \subseteq \{ 1, \dots, \ell (\bp) \}$ and $J_0 \subseteq \{ 1, \ldots, \ell (\bq) \}$, and a bijection $u \colon J_0 \to I_0$ such that the sets $\{ p_i \mid i \in \{ 1, \ldots, \ell (\bp) \} \setminus I_0 \}$ and $\{ q_j \mid j \in \{ 1, \ldots, \ell (\bq) \} \setminus J_0 \}$ have no common elements, and $p_{u (j)} = q_j$ for each $j \in J_0$. Next we construct partial injective maps
$$
v_+, v_- \colon \{ 1, \ldots, \ell (\bq) \} \dashrightarrow \{ 1, \ldots, \ell (\bp) \}
$$
with the following properties:
\begin{enumerate}

\item
$v_+ (j)$ is defined if and only if $j \notin J_0$ and the set
\[
I_j^+ := \{ i \in \{ 1, \ldots, \ell (\bp) \} \setminus I_0 \mid p_i > q_j \} \setminus v_+ (\{ 1, \ldots, j - 1 \})
\]
is nonempty; moreover, if this is the case, then
$$
v_+ (j) := \min I_j^+,
$$
i.e. $v_+ (j)$ is the smallest index $i$ not in $I_0$ such that $p_i > q_j$, and which is not attached to an index smaller then $j$, if it exists; in particular, we construct $v_+$ by increasing induction on $j$.

\item
$v_- (j)$ is defined if and only if $j \notin J_0$ and the set
\[
I_j^- := \{ i \in \{ 1, \ldots, \ell (\bp)] \setminus I_0 \mid p_i < q_j \} \setminus v_- (\{ j + 1, \ldots, \ell (\bq) \})
\]
is nonempty; moreover, if this is the case, then
$$
v_- (j) := \max I_j^-,
$$
i.e. $v_- (j)$ is the biggest index $i$ not in $I_0$ such that $p_i < q_j$, and which is not attached to an index smaller then $j$, if it exists; in particular, we construct $v_-$ by decreasing induction on $j$.

\end{enumerate}
Let
$$
v_+^{-1}, v_-^{-1} \colon \{ 1, \ldots, \ell (\bp) \} \dashrightarrow \{ 1, \ldots, \ell (\bq) \}
$$ be the (partial) inverse maps. The following pairs of partitions constitute the canonical decomposition of $(\bp, \bq)$. First, for each $j \in J_0$, we have the pair $((p_{u (j)}), (q_j))$. Next, if $i \in \{ 1, \ldots, \ell (\bp) \} \setminus (I_0 \cup v_- (\{ 1, \ldots, \ell (\bq)\}))$, then the corresponding pair is given by the partitions
\[
(p_i, p_{v_- v_+^{-1} (i)}, p_{v_- v_+^{-1} v_- v_+^{-1} (i)}, \ldots) \qquad \text{and} \qquad (q_{v_+^{-1} (i)}, q_{v_+^{-1} v_- v_+^{-1} (i)}, \ldots)
\]
(note that the second partition may be empty). Dually, if $j \in \{ 1, \ldots, \ell (\bq) \} \setminus (I_0 \cup v_+^{-1} (\{ 1, \ldots, \ell (\bp)\}))$, then the partitions are
\[
(p_{v_- (j)}, p_{v_- v_+^{-1} v_- (j)}, \ldots) \qquad \text{and} \qquad (q_j, q_{v_+^{-1} v_- (j)}, q_{v_+^{-1} v_- v_+^{-1} v_- (j)}, \ldots).
\]
Note that (up to ordering) the obtained partitions do not depend on a choice of the sets $I_0$ and $J_0$. If the pairs $(\bp^{(1)}, \bq^{(1)})$, \ldots, $(\bp^{(k)}, \bq^{(k)})$ form the canonical decomposition of $(\bp, \bq)$, then we write
\[
(\bp, \bq) = (\bp^{(1)}, \bq^{(1)}) \oplus \cdots \oplus (\bp^{(k)}, \bq^{(k)}).
\]
If this is the case, then
\[
(\bp, \bq) = (\bp^{(1)}, \bq^{(1)}) \cup \cdots \cup (\bp^{(k)}, \bq^{(k)}),
\]
and the pairs $(\bp^{(1)}, \bq^{(1)})$, \ldots, $(\bp^{(k)}, \bq^{(k)})$ are indecomposable. On the other hand, if the pair $(\bp, \bq)$ is indecomposable, then its canonical decomposition consists of a single pair $(\bp, \bq)$.

If
\[
(\bp, \bq) = (\bp^{(1)}, \bq^{(1)}) \oplus \cdots \oplus (\bp^{(k)}, \bq^{(k)}),
\]
then (by abuse of notation) we put
\[
M_{\bp, \bq} := M_{\bp^{(1)}, \bq^{(1)}} \oplus \cdots \oplus M_{\bp^{(k)}, \bq^{(k)}}.
\]
Note that Lemma~\ref{lemma iso} guarantees that this definition coincides (up to isomorphism) with the one from Subsection~\ref{sub weakly}, if $(\bp, \bq)$ is weakly indecomposable. Indeed, let
\[
(\bp, \bq) = (\bp^{(0)}, \bq^{(0)}) \oplus (\bp^{(1)}, \bq^{(1)}) \cdots \oplus (\bp^{(k)}, \bq^{(k)})
\]
be the canonical decomposition of $(\bp, \bq)$. Since $(\bp, \bq)$ is weakly indecomposable, we may order the summands in such a way that, for each $i = 1, \ldots, k$, $\bp^{(i)} = (n_i) = \bq^{(i)}$ for some $0 < n_i \leq m$. Consequently, by Lemma~\ref{lemma iso} we get
\[
M_{\bp^{(0)}, \bq^{(0)}} \oplus M_{\bp^{(1)}, \bq^{(1)}} \oplus \cdots \oplus M_{\bp^{(k)}, \bq^{(k)}} \cong M_{\bp, \bq}
\]
(recall, $(\bp^{(0)}, \bq^{(0)}) \cup (\bp^{(1)}, \bq^{(1)}) \cup \cdots \cup (\bp^{(k)}, \bq^{(k)}) = (\bp, \bq)$).

\subsection{Examples}
We present now examples illustrating definitions from Subsections~\ref{sub weakly} and~\ref{subsect can}.

\subsubsection{}
Let $\bp := (6, 3, 2)$ and $\bq := (4, 2, 1)$. Note that the pair $(\bp,\bq)$ is weakly indecomposable. Then $M_{\bp,\bq} = (U_{\bp},U_{\bq},h_{\bp,\bq})$ looks as displayed in Figure~\ref{fig:Mpq}.
\begin{figure}[h]
$$
\xymatrix@-0.3cm{
0 \ar[d] &
\\
0 \ar[d] &
\\
0 \ar[d] & 1\ar[l]\ar[r]\ar[d] & 0 \ar[d]
\\
0 \ar[d] & 1\ar[l]\ar[r]\ar[d] & 0 \ar[d] & 1\ar[l]\ar[d]\ar[r] &
0 \ar[d]
\\
0 \ar[d] & 1\ar[l]\ar[r]\ar[d] & 0 & 1\ar[l]\ar[r] &
0 & 1 \ar[l]
\\
0  & 1\ar[l] &
}
$$
\caption{The module $M_{(6,3,2),(4,2,1)}$.} \label{fig:Mpq}
\end{figure}
Let us explain this in more detail. Each of the numbers $0$ and $1$ appearing in Figure~\ref{fig:Mpq} stands for a basis vector of $U_\bp$ and $U_\bq$, respectively. Thus $\dim(U_{\bp}) = 11$ and $\dim(U_{\bq}) = 7$. The arrows show how the generators $\varepsilon_0 := \left[
\begin{smallmatrix}
X & 0 \\ 0 & 0
\end{smallmatrix}
\right]$, $\varepsilon_1 := \left[
\begin{smallmatrix}
0 & 0 \\ 0 & X
\end{smallmatrix}
\right]$ and $\alpha := \left[
\begin{smallmatrix}
0 & 1 \\ 0 & 0
\end{smallmatrix}
\right]$ of $A$ act on the basis vectors. For example
\[
\xymatrix@-0.3cm{0 \ar[d]\\0}
\]
means that $\varepsilon_0$ applied to the basis vector labelled by the upper $0$ is equal to the basis vector labelled by the lower $0$, while
\[
\xymatrix@-0.3cm{0 & 1 \ar[l]\ar[r] & 0}
\]
means that $\alpha$ applied to the basis vector labelled by $1$ is equal to the sum of the two basis vectors labelled by the $0$'s.

\subsubsection{} \label{example can}
Let
\begin{align*}
\bp & = (19, 18, 17, 16, 13, 13, 10, 10, 9, 6, 6, 2, 2, 1),
\\
\bq & = (19, 15, 14, 13, 13, 13, 12, 8, 4, 4, 3, 2).
\end{align*}
We write the entries of $\bp$ and $\bq$ in two rows, indicating the ordering between them:
\[
\makebox[0pt][c]{\xymatrix@C=4pt{
19 & 18 & 17 & 16 & & & 13 & 13 & & & 10 & 10 & 9 & & 6 & 6 & & & & 2 & 2 & 1
\\
19 & & & & 15 & 14 & 13 & 13 & 13 & 12 & & & & 8 & & & 4 & 4 & 3 & 2
}}
\]
Then, we draw the maximal set of vertical (connecting identical entries, disjoint) edges, and we get
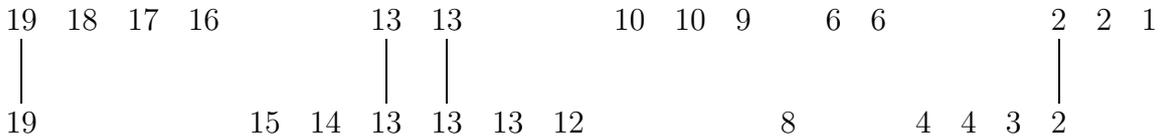
\[
\makebox[0pt][c]{\xymatrix@C=4pt{
19 & 18 & 17 & 16 & & & 13 & 13 & & & 10 & 10 & 9 & & 6 & 6 & & & & 2 & 2 & 1
\\
19 \ar@{-}[u] & & & & 15 & 14 & 13 \ar@{-}[u] & 13 \ar@{-}[u] & 13 & 12 & & & & 8 & & & 4 & 4 & 3 & 2 \ar@{-}[u]
}}
\]
Next, for each lower entry which is not connected to a vertical edge, we draw, starting from the left, an edge which connects its to the first upper entry on the left, which is not yet connected to another edge (if there is not such an entry, we do not draw an edge starting at the given entry). Thus, we first draw an edge connecting the lower 15 with the upper 16, then an edge connecting the lower 14 with the upper 17, and an edge connecting the free lower 13 with the upper 18:
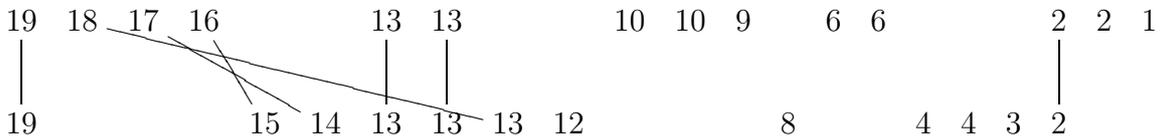
\[
\makebox[0pt][c]{\xymatrix@C=4pt{
19 & 18 & 17 & 16 & & & 13 & 13 & & & 10 & 10 & 9 & & 6 & 6 & & & & 2 & 2 & 1
\\
19 \ar@{-}[u] & & & & 15 \ar@{-}[lu] & 14 \ar@{-}[lllu] & 13 \ar@{-}[u] & 13 \ar@{-}[u] & 13 \ar@{-}[lllllllu] & 12 & & & & 8 & & & 4 & 4 & 3 & 2 \ar@{-}[u]
}}
\]
We cannot draw an edge starting at the lower $12$ (since there is no free upper entry left to its left). Continuing in this way we get
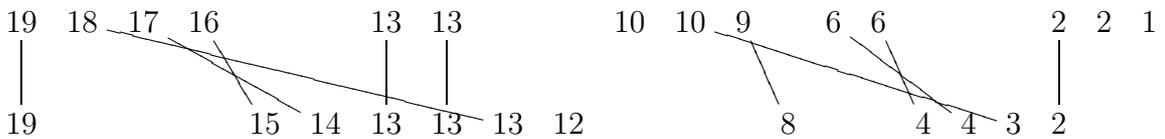
\[
\makebox[0pt][c]{\xymatrix@C=4pt{
19 & 18 & 17 & 16 & & & 13 & 13 & & & 10 & 10 & 9 & & 6 & 6 & & & & 2 & 2 & 1
\\
19 \ar@{-}[u] & & & & 15 \ar@{-}[lu] & 14 \ar@{-}[lllu] & 13 \ar@{-}[u] & 13 \ar@{-}[u] & 13 \ar@{-}[lllllllu] & 12 & & & & 8 \ar@{-}[lu] & & & 4 \ar@{-}[lu] & 4 \ar@{-}[lllu] & 3 \ar@{-}[lllllllu] & 2 \ar@{-}[u]
}}
\]
Finally, for each lower entry which is not connected to a vertical edge, we draw, starting from the right, an edge which connects its to the first upper entry on the right, which is connected neither to a vertical edge nor to an SW-NE edge drawn earlier (again, if there is not such entry, we do not draw an edge):
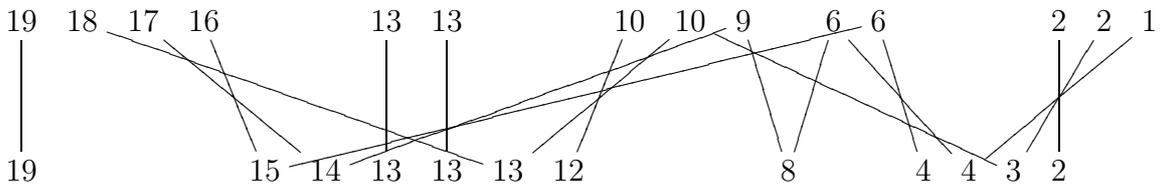
\[
\makebox[0pt][c]{\xymatrix@C=4pt@R=3\baselineskip{
19 & 18 & 17 & 16 & & & 13 & 13 & & & 10 & 10 & 9 & & 6 & 6 & & & & 2 & 2 & 1
\\
19 \ar@{-}[u] & & & & 15 \ar@{-}[lu] \ar@{-}[rrrrrrrrrrru] & 14 \ar@{-}[lllu] \ar@{-}[rrrrrrru] & 13 \ar@{-}[u] & 13 \ar@{-}[u] & 13 \ar@{-}[lllllllu] \ar@{-}[rrru] & 12 \ar@{-}[ru] & & & & 8 \ar@{-}[lu] \ar@{-}[ru] & & & 4 \ar@{-}[lu] & 4 \ar@{-}[lllu] \ar@{-}[rrrru] & 3 \ar@{-}[lllllllu] \ar@{-}[rru] & 2 \ar@{-}[u]
}}
\]
The connected components of the obtained graph give the canonical decomposition of $(\bp, \bq)$, i.e.
\begin{multline*}
(\bp, \bq) = ((19), (19)) \oplus ((13), (13)) \oplus ((13), (13)) \oplus ((2), (2)) \oplus ((18, 10, 2), (13, 3))
\\
\oplus ((17, 9, 6, 1), (14, 8, 4)) \oplus ((16, 6), (15, 4)) \oplus ((10), (12)).
\end{multline*}


\section{Proof of Theorem~\ref{thm:intro4}}


\subsection{Stratification and general representations}
The main tool which we use in the proof of Theorem~\ref{thm:intro4} is the stratification by pairs of partitions, which we introduce now. If $d_0, d_1 \geq 0$ and $(\bp, \bq) \in \cP_m (d_0, d_1) := \cP_m (d_0) \times \cP_m (d_1)$, then we define
\[
\cS_{\bp, \bq} := \{ M \in \cX_m (d_0, d_1) \mid \text{$M_0 \cong U_\bp$ and $M_1 \cong U_\bq$} \}.
\]
Then $\cS_{\bp, \bq}$ is a locally closed irreducible subset of $\cX_m (d_0, d_1)$ (see~\cite{BobinskiSchroer2017a}*{Subsection~5.3}).
Consequently, if $\cY$ is an irreducible component of $\cX_m (d_0, d_1)$, then there exists a (uniquely determined) pair $(\bp, \bq)$ of partitions such that $\cY = \ov{\cS_{\bp, \bq}}$. For semicontinuity reasons (see also the last paragraph of Subsection~\ref{sub geometry}), if $\bp$ and $\bq$ are both maximal in $\cP_m(d_0)$ and $\cP_m(d_1)$, respectively, then $\ov{\cS_{\bp,\bq}}$ is an irreducible component of $\cX_m (d_0, d_1)$. Our ultimate goal is to prove that this is the only irreducible component of $\cX_m (d_0, d_1)$, or equivalently that $\ov{\cS_{\bp,\bq}} = \cX_m (d_0, d_1)$.

For an $A$-module $M$ we denote by $\cO_M$ the subset of $\cX_m (d_0, d_1)$ consisting of $N \in \cX_m (d_0, d_1)$ such that $N \cong M$ (as $A$-modules), where $d_0 := \dim_K M_0$ and $d_1 := \dim_K M_1$. One can easily check that $\cO_M$ is the $G_{d_0,d_1}$-orbit of $M$. The aim of this subsection is to prove the following result, which has been obtained earlier by Lutz Hille and Dieter Vossieck with different methods. Unfortunately, their proof has not been published yet, hence we include ours for completion.

\begin{Prop} \label{prop generic}
If $d_0, d_1 \geq 0$ and $(\bp, \bq) \in \cP_m (d_0, d_1)$, then
\[
\ov{\cS_{\bp, \bq}} = \ov{\cO_{M_{\bp, \bq}}}.
\]
\end{Prop}

We use the notation from Subsection~\ref{subsect can}. Let $M$ be a generic module in $\cS_{\bp, \bq}$. Obviously we may assume $M_0 = U_\bp$ and $M_1 = U_\bq$. We show there exist automorphisms $f (0) \in \Aut_\Lambda (U_\bp)$ and $f (1) \in \Aut_\Lambda (U_\bq)$ such that, for $\Phi = (\phi_{i j}) := f (0) h_M f (1)^{-1}$, the following conditions are satisfied:
\begin{enumerate}

\item
if $j_0 \in J_0$ and $i_0 := u (j_0)$, then $\phi_{i_0 j_0} = 1$, $\phi_{i j_0} = 0$, for each $i \neq i_0$, and $\phi_{i_0 j} = 0$, for each $j \neq j_0$;

\item
if $j_0 \notin J_0$, $i \neq i_0^+ := v_+ (j_0)$, and $p_i > q_{j_0}$, then $\phi_{i j_0} = 0$; moreover, $\phi_{i_0^+ j_0} = X^{p_{i_0^+} - q_{j_0}}$ (if $i_0^+$ is defined);

\item
if $j_0 \notin J_0$, $i \neq i_0^- := v_- (j_0)$, and $p_i < q_{j_0}$, then $\phi_{i j_0} = 0$; moreover, $\phi_{i_0^- j_0} = 1$ (if $i_0^-$ is defined).

\end{enumerate}
In other words,
\[
M \cong M_{\bp^{(1)}, \bq^{(1)}} \oplus \cdots \oplus M_{\bp^{(k)}, \bq^{(k)}},
\]
where $(\bp, \bq) = (\bp^{(1)}, \bq^{(1)}) \oplus \cdots \oplus (\bp^{(k)}, \bq^{(k)})$ is the canonical decomposition of $(\bp, \bq)$.
This will imply our claim. We prove the above in a number of steps.

\begin{step} \label{step new1}
There exist automorphisms $f (0) \in \Aut_\Lambda (U_\bp)$ and $f (1) \in \Aut_\Lambda (U_\bq)$ such that, for each $j_0 \in J_0$, the following conditions are satisfied:
\begin{enumerate}
\renewcommand{\theenumi}{\alph{enumi}}

\item
$\phi_{i_0 j_0} = 1$, where $i_0 := u (j_0)$;

\item
$\phi_{i j_0} = 0$, for each $i \neq i_0$;

\item
$\phi_{i_0 j} = 0$, for each $j \neq j_0$;

\end{enumerate}
where $\Phi = (\phi_{i j}) := f (0) h_M f (1)^{-1}$.
\end{step}

\begin{proof}[Proof of Step~\ref{step new1}]
Let $\Psi = (\psi_{i j}) := h_M$. By genericity of $M$ we may assume that $\psi_{i_0 j_0}$ is an isomorphism. By multiplying the $i_0$-th row of $\Psi$ by the inverse of $\psi_{i_0 j_0}$, we may assume $\psi_{i_0 j_0} = 1$. Obviously, $\psi_{i j_0}$ ($\psi_{i_0 j}$) factors through $\psi_{i_0 j_0}$, for each $i \neq i_0$ ($j \neq j_0$, respectively). By performing appropriate row (column, respectively) operations on $\Psi$, we may assume these maps are zero, and the claim follows. We also remark that neither the $i_0$-th row nor the $j_0$-th column will be involved in elementary operations in the remaining steps of the proof.
\end{proof}

Step~\ref{step new1} implies that we are dealing with the matrix of the form
\[
\begin{bmatrix}
\Id_{\bigoplus_{j \in J_0} U_{(p_j)}} & 0
\\
0 & \Psi'
\end{bmatrix}.
\]
In the rest of the proof we concentrate on the matrix $\Psi'$. Thus in order to simplify the presentation we assume $I_0 = \emptyset = J_0$, i.e.\ $h_M = \Psi'$.

\begin{step} \label{step new2}
There exist automorphisms $f (0) \in \Aut_\Lambda (U_\bp)$ and $f (1) \in \Aut_\Lambda (U_\bq)$ such that, for each $j_0$, the following condition is satisfied:
\begin{enumerate}
\renewcommand{\theenumi}{\alph{enumi}}

\item
if $i_0^+ := v_+ (j_0)$ (in particular it means that $v_+ (j_0)$ is defined), then $\phi_{i_0^+ j_0}$ is a monomorphism and $\phi_{i_0^+ j} = 0$ for each $j > j_0$;

\end{enumerate}
where $\Phi = (\phi_{i j}) := f (0) h_M f (1)^{-1}$.
\end{step}

\begin{proof}[Proof of Step~\ref{step new2}]
Let $\Psi = (\psi_{i j})$ be the matrix obtained in Step~\ref{step new1}. By induction on $j_0$ we show we may assume $\psi_{i_0^+ j_0}$ is a monomorphism, where $i_0^+ := v_+ (j_0)$, and $\psi_{i_0^+ j} = 0$ for each $j > j_0$. Indeed, we may assume $\psi_{i_0^+ j_0}$ is a monomorphism by genericity. Moreover, take $j > j_0$. Then, $\psi_{i_0^+ j} = \psi_{i_0 j_0} \circ \varphi$ for some $\varphi \colon U_{(q_j)} \to U_{(q_{j_0})}$ by the factorization property described in Subsection~\ref{sub truncated}, since $q_j \leq q_{j_0}$. Consequently, if we subtract the $j_0$-th column composed on the right with $\varphi$ from the $j$-th column, we may assume $\psi_{i_0^+ j} = 0$. Note that if $j_1 < j_0$ and $i_1^+ := v_+ (j_1)$, then by induction assumption $\psi_{i_1^+ j_0} = 0$. Consequently, the entry $\psi_{i_1^+ j}$, which is $0$ by an earlier induction step, is not changed by this column operation.
\end{proof}

\begin{step} \label{step new3}
There exist automorphisms $f (0) \in \Aut_\Lambda (U_\bp)$ and $f (1) \in \Aut_\Lambda (U_\bq)$ such that, for each $j_0$, the following condition is satisfied:
\begin{enumerate}
\renewcommand{\theenumi}{\alph{enumi}}

\item
if $i \neq i_0^+ := v_+ (j_0)$ and $p_i > q_{j_0}$, then $\phi_{i j_0} = 0$; moreover, if $i_0^+$ is defined, then $\phi_{i_0^+ j_0}$ is a monomorphism;

\end{enumerate}
where $\Phi = (\phi_{i j}) := f (0) h_M f (1)^{-1}$.
\end{step}

\begin{proof}[Proof of Step~\ref{step new3}]
Let $\Psi = (\psi_{i j})$ be the matrix obtained in Step~\ref{step new2}. We proceed by decreasing induction on $j_0$. Take $i \neq i_0^+ := v_+ (j_0)$ with $p_i > q_{j_0}$. If $i = v_+ (j_1)$, for some $j_1 < j_0$, then $\psi_{i j_0} = 0$ by Step~\ref{step new2}, hence nothing has to be done in this case. Note that in particular this implies that the entries in the rows indexed by the elements of $v_+ (\{ 1, \ldots, j_0 - 1 \})$ are not changed.

Assume $i \not \in v_+ (\{ 1, \ldots, j_0 - 1 \})$. Then $i_0^+$ is defined, $\psi_{i_0^+ j_0}$ is a monomorphism by Step~\ref{step new2}, and $\psi_{i j_0} = \varphi \circ \psi_{i_0^+ j_0}$ for some $\varphi \colon U_{(p_{i_0^+})} \to U_{(p_i)}$ (again by the factorization property), since by the construction of $i_0^+$, $i > i_0^+$, thus $p_i \geq p_{i_0^+}$. We subtract the $i_0^+$-th row composed on the left with $\varphi$ from the $i$-th row, hence we may assume $\psi_{i j_0} = 0$. We need to observe the following. If $j_1 > j_0$ (in particular, $p_i > q_{j_0} \geq q_{j_1}$), then the entry $\psi_{i j_1}$ does not change, since $\psi_{i_0^+ j_1} = 0$ by induction assumption (note that $p_{i_0^+} > q_{j_0} \geq q_j$). 
\end{proof}

\renewcommand{\thestep}{\Roman{step}*}
\setcounter{step}{1}

\begin{step} \label{step new2prime}
There exist automorphisms $f (0) \in \Aut_\Lambda (U_\bp)$ and $f (1) \in \Aut_\Lambda (U_\bq)$ such that, for each $j_0$, the following condition are satisfied:
\begin{enumerate}
\renewcommand{\theenumi}{\alph{enumi}}

\item
if $i \neq i_0^+ := v_+ (j_0)$ and $p_i > q_{j_0}$, then $\phi_{i j_0} = 0$; moreover, if $i_0^+$ is defined, then $\phi_{i_0^+ j_0}$ is a monomorphism;

\item
if $i_0^- := v_- (j_0)$ (in particular it means that $v_- (j_0)$ is defined), then $\phi_{i_0^- j_0}$ is an epimorphism and $\phi_{i_0^- j} = 0$ for each $j < j_0$;

\end{enumerate}
where $\Phi = (\phi_{i j}) := f (0) h_M f (1)^{-1}$.
\end{step}

\begin{proof}[Proof of Step~\ref{step new2prime}]
Let $\Psi = (\psi_{i j})$ be the matrix obtained in Step~\ref{step new3}. The proof is dual to the proof of Step~\ref{step new2}. The crucial step consists of adding a multiplicity of the $j_0$-th column to the $j$-th column, where $j < j_0$. We additionally have to check, that if $p_i > q_j$, then $\psi_{i j}$ does not change. However, if $p_i > q_j$, then $p_i > q_{j_0}$, hence $\psi_{i j_0} = 0$ and the claim follows.
\end{proof}

\begin{step} \label{step new3prime}
There exist automorphisms $f (0) \in \Aut_\Lambda (U_\bp)$ and $f (1) \in \Aut_\Lambda (U_\bq)$ such that, for each $j_0$, the following conditions are satisfied:
\begin{enumerate}
\renewcommand{\theenumi}{\alph{enumi}}

\item
if $i \neq i_0^+ := v_+ (j_0)$ and $p_i > q_{j_0}$, then $\phi_{i j_0} = 0$; moreover, if $i_0^+$ is defined, then $\phi_{i_0^+ j_0}$ is a monomorphism;

\item
if $i \neq i_0^- := v_+ (j_0)$ and $p_i < q_{j_0}$, then $\phi_{i j_0} = 0$; moreover, if $i_0^-$ is defined, then $\phi_{i_0^- j_0}$ is an epimorphism;

\end{enumerate}
where $\Phi = (\phi_{i j}) := f (0) h_M f (1)^{-1}$.
\end{step}

\begin{proof}[Proof of Step~\ref{step new3prime}]
The proof is mostly dual to the proof of Step~\ref{step new3}, hence we leave it to the reader.
\end{proof}

\begin{proof}[The final step of the proof of Proposition~\ref{prop generic}]
Let $\Phi = (\phi_{i j}) $ be the matrix obtained in Step~\ref{step new3prime}. We have to show that, in addition to what we did so far, we may assume $\phi_{i_0^+ j_0} = X^{p_{i_0^+} - q_{j_0}}$ and $\phi_{i_0^- j_0} = 1$ every time it makes sense, where $i_0^+ := v_+ (j_0)$ and $i_0^- := v_- (j_0)$. We prove this by induction on  the depth of $j_0$, where we define the depth of $j_0$ to be the maximal $l$ such that the element $(v_-^{-1} v_+)^l (j_0)$ is defined. If we fix $j_0$, then $\phi_{i_0^+ j_0} = X^{p_{i_0^+} - q_{j_0}} \circ \varphi$, where $\varphi \in \Aut_\Lambda (U_{(q_{j_0})})$, hence if we multiply (on the right) the $j_0$-th column of $\Phi$ by $\varphi^{-1}$, we may assume $\phi_{i_0^+ j_0} = X^{p_{i_0^+} - q_{j_0}}$. Next, $\phi_{i_0^- j_0} = \psi \circ 1$, for some $\psi \in \Aut_\Lambda (U_{(p_{i_0^-})})$, and we multiply the $i_0^-$-th row of $\Phi$ by $\psi^{-1}$. Note that in this way we do not change columns, whose index has depth smaller or equal the depth of $j_0$, different from the $j_0$-th column. This finishes the proof.
\end{proof}

\subsection{Indecomposable irreducible components}
In what follows we use the following convention. If $\bq$ is a partition, then $q_0 := m$. Thus the condition $q_1 < q_0$ reads as $q_1 < m$. Similarly, if $i \geq \ell (\bq)$, then the condition $q_i > q_{i + 1}$ means $q_i > 0$.

\begin{Lem} \label{lem def ind}
Let $(\bp, \bq)$ be a weakly indecomposable pair of partitions. Assume $q_{j_1} < q_{j_1 - 1}$ and $q_{j_2} > q_{j_2 + 1}$, for some $j_1 < j_2$. Let $\bq'$ be defined by
\[
q_j' :=
\begin{cases}
q_j + 1 & \text{if $j = j_1$},
\\
q_j - 1 & \text{if $j = j_2$},
\\
q_j & \text{otherwise}.
\end{cases}
\]
If $(\bp, \bq')$ is weakly indecomposable, then $\cS_{\bp, \bq} \subseteq \ov{\cS_{\bp, \bq'}}$.
\end{Lem}

\begin{proof}
Note that $(\bp, \bq)$ and $(\bp, \bq')$ are weakly indecomposable of the same type. Indeed, it $(\bp, \bq)$ is of mono type, then $p_{j_2} \geq q_{j_2} > q_{j_2}'$, hence $(\bp, \bq')$ cannot be of epi type, thus it has to be of mono type. We argue similarly if $(\bp, \bq)$ is of epi type. Note that it also follows from the above that both $(\bp, \bq)$ and $(\bp, \bq')$ cannot be both simultaneously of mono and epi type.

Using Proposition~\ref{prop generic}, it is enough to show $\cO_{M_{\bp, \bq}} \subseteq \ov{\cO_{M_{\bp, \bq'}}}$. In order to prove this it is sufficient to construct an exact sequence
\[
0 \to M_{\bp, \bq} \xrightarrow{\left[
\begin{smallmatrix}
f_1 \\ f_2
\end{smallmatrix}
\right]} M_{\bp, \bq'} \oplus N \xrightarrow{\left[
\begin{smallmatrix}
g_1 & g_2
\end{smallmatrix}
\right]}
N \to 0
\]
(see~\cite{Riedtmann}*{Proposition~3.4}). Before we describe the module $N$ and the maps $f_1$, $f_2$, $g_1$ and $g_2$, we fix some notation.

Let $i_1 := j_1 + 1$, if the pairs are of mono type, and $i_1 := j_1$, if the pairs are of epi type. Similarly, $i_2 := j_2$, if the pairs are of mono type, and $i_2 := j_2 - 1$, if the pairs are of epi type. Let
\[
U' := \bigoplus_{i < i_1} U_{(p_i)}, \qquad U := \bigoplus_{i_1 \leq i \leq i_2} U_{(p_i)}, \qquad \text{and} \qquad U'' := \bigoplus_{i > i_2} U_{(p_i)}.
\]
Similarly,
\[
V' := \bigoplus_{j < j_1} U_{(q_j)}, \qquad V := \bigoplus_{j_1 < j < i_2} U_{(q_j)}, \qquad \text{and} \qquad V'' := \bigoplus_{j > j_2} U_{(q_j)}.
\]
Then
\begin{gather*}
U_\bp = U' \oplus U \oplus U'',
\\
U_\bq = V' \oplus U_{(q_{j_1})} \oplus V \oplus U_{(q_{j_2})} \oplus V''
\\
\intertext{and}
U_\bq' = V' \oplus U_{(q_{j_1} + 1)} \oplus V \oplus U_{(q_{j_2} - 1)} \oplus V''.
\end{gather*}
We keep these decompositions till the end of the proof. In particular, with respect to these decompositions
\[
h_{\bp, \bq} =
\begin{bmatrix}
\varphi_{1, 1} & \varphi_{1, 2} \circ X & 0 & 0 & 0
\\
0 & \varphi_{2, 2} & \varphi_{2, 3} & \varphi_{2, 4} & 0
\\
0 & 0 & 0 & \varphi_{3, 4} \circ 1 & \varphi_{3, 5}
\end{bmatrix}
\]
and
\[
h_{\bp, \bq'} =
\begin{bmatrix}
\varphi_{1, 1} & \varphi_{1, 2} & 0 & 0 & 0
\\
0 & \varphi_{2, 2} \circ 1 & \varphi_{2, 3} & \varphi_{2, 4} \circ X & 0
\\
0 & 0 & 0 & \varphi_{3, 4}  & \varphi_{3, 5}
\end{bmatrix}.
\]
It is important to observe that $\begin{bmatrix} \varphi_{2, 3} & \varphi_{2, 4} \end{bmatrix}$ is a monomorphism, while $\begin{bmatrix} \varphi_{2, 2} \circ 1 & \varphi_{2, 3} \end{bmatrix}$ is an epimorphism.

We put $N_0 := U =: N_1$ and $h_N := \Id_U$. Next
\[
f_1 (0) :=
\begin{bmatrix}
\Id_{U'} & 0 & 0
\\
0 & X \cdot \Id_U & 0
\\
0 & 0 & \Id_{U''}
\end{bmatrix}
\quad \text{and} \quad
f_1 (1) :=
\begin{bmatrix}
\Id_{V'} & 0 & 0 & 0 & 0
\\
0 & X & 0 & 0 & 0
\\
0 & 0 & X \cdot \Id_V & 0 & 0
\\
0 & 0 & 0 & 1 & 0
\\
0 & 0 & 0 & 0 & \Id_{V'}
\end{bmatrix}.
\]
Similarly,
\[
f_2 (0) :=
\begin{bmatrix}
0 & \Id_U & 0
\end{bmatrix}
\qquad \text{and} \qquad
f_2 (1) :=
\begin{bmatrix}
0 & \varphi_{2, 2} & \varphi_{2, 3} & \varphi_{2, 4} & 0
\end{bmatrix}.
\]
Further,
\[
g_1 (0) :=
\begin{bmatrix}
0 & \Id_U & 0
\end{bmatrix}
\qquad \text{and} \qquad
g_1 (1) :=
\begin{bmatrix}
0 & \varphi_{2, 2} \circ 1 & \varphi_{2, 3} & \varphi_{2, 4} \circ X & 0
\end{bmatrix}.
\]
Finally,
\[
g_2 (0) := - X \cdot \Id_U =: g_2 (1).
\]
We leave it to the reader to verify, that the sequence obtained in this way is actually an exact sequence in the category of $A$-modules.
\end{proof}

Using Lemma~\ref{lem def ind} and its dual, we get the following.

\begin{Cor} \label{cor ind irr}
Let $d_0, d_1 \geq 0$ and $(\bp, \bq) \in \cP_m (d_0, d_1)$ be an indecomposable pair of partitions, such that $\ov{\cS_{\bp, \bq}}$ is an irreducible component of $\cX_m (d_0, d_1)$. Then $(\bp, \bq)$ is one of the following pairs:
\begin{enumerate}

\item
$((p), (q))$;

\item
$((m, p), (q))$ with $m > q > p > 0$;

\item
$((p), (m, q))$ with $m > p >  q > 0$.
\end{enumerate}
\end{Cor}

If $(\bp, \bq)$ is an indecomposable pair of partitions such that $\ov{\cS_{\bp, \bq}}$ is an irreducible component of $\cX_m (d_0, d_1)$, then we call $\ov{\cS_{\bp, \bq}}$ an indecomposable irreducible component. Recall from Lemma~\ref{lemma ind} that if this is the case, then $M_{\bp, \bq}$ is an indecomposable $A$-module.

\subsection{Deformations between strata}
In addition to Lemma~\ref{lem def ind} we have the following result about deformations of strata. The author thanks Grzegorz Zwara for a remark, which helped improve the formulation of the lemma and simplify its proof.

\begin{Lem} \label{lem def strata}
Let $d_0, d_1 \geq 0$, $\bp \in \cP_m (d_0)$, and $\bq \in \cP_M (d_1)$. Assume there exist $j_1 \leq j_2$ such that
\[
q_{j_1 - 1} > q_{j_1}, \qquad q_{j_2} > q_{j_2 + 1},
\]
and for each $i$, either $p_i > q_{j_1}$ or $p_i < q_{j_2}$. If the partition $\bq'$ is defined by
\[
q_j' :=
\begin{cases}
q_j + 1 & \text{if $j = j_1$},
\\
q_j - 1 & \text{if $j = j_2$},
\\
q_j & \text{otherwise},
\end{cases}
\]
then $\cS_{\bp, \bq} \subseteq \ov{\cS_{\bp, \bq'}}$.
\end{Lem}

\begin{proof}
We use the following general result. Let $\pi \colon \cY \to \cX$ be a vector bundle with $\cX$ irreducible. If $\cU$ is a nonempty open subset of $\cX$, then $\pi^{-1} (\cU)$ is a dense subset of $\cY$.

In our case, let $\cX$ be the set of pairs $(M_0, M_1)$ such that $M_0 \in \ov{\cO_\bp}$, $M_1 \in \ov{\cO_{\bq'}}$ (see Subsection~\ref{sub geometry}), and
\[
\dim_K \Hom_\Lambda (M_1, M_0) = \dim_K \Hom_\Lambda (U_{\bq'}, U_{\bp}).
\]
Next, $\cY$ is the set of $M \in \cX_m (d_0, d_1)$ such that $(M_0, M_1) \in \cX$. The hom-condition implies that the natural projection $\pi \colon \cY \to \cX$ is a vector bundle. Moreover, $\cU := \cO_\bp \times \cO_{\bq'}$ is an open subset of $\cX$ and $\pi^{-1} (\cU) = \cS_{\bp, \bq'}$. Finally, observe that our assumptions imply $\cO_\bp \times \cO_\bq \subseteq \cX$, hence $\cS_{\bp, \bq} = \pi^{-1} (\cO_\bp \times \cO_\bq) \subseteq \ov{\cS_{\bp, \bq'}}$ by the result mentioned at the beginning of the proof.
\end{proof}

Obviously, there is the dual version of Lemma~\ref{lem def strata} (we change $\bp$ instead of $\bq$).

\subsection{Direct sums of indecomposable irreducible components}
We study now when the direct sum of two indecomposable irreducible components is an irreducible component.

\begin{Lem} \label{lem direct sums}
Let $(\bp, \bq)$ and $(\br, \bs)$ be pairs of partitions such that $\ov{\cS_{\bp, \bq}}$ and $\ov{\cS_{\br, \bs}}$ are indecomposable irreducible components. Then $\cS_{\bp, \bq} \oplus \cS_{\br, \bs}$ is not an irreducible component, unless one of the following conditions is satisfied:
\begin{enumerate}

\item \label{cond one}
$(\bp, \bq) = ((m), (m))$, or

\item \label{cond two}
$(\br, \bs) = ((m), (m))$, or

\item \label{cond three}
$(\bp, \bq) = ((m), (0))$ and $(\br, \bs)$ is of mono type, or

\item \label{cond four}
$(\bp, \bq) = ((0), (m))$ and $(\br, \bs)$ is of epi type, or

\item \label{cond five}
$(\br, \bs) = ((m), (0))$ and $(\bp, \bq)$ is of mono type, or

\item \label{cond six}
$(\br, \bs) = ((0), (m))$, and $(\bp, \bq)$ is of epi type.

\end{enumerate}
\end{Lem}

\begin{proof}
There is a number of cases we have to consider. We will number them in order to make it easier to follow the proof. Note that according to Corollary~\ref{cor ind irr} $(\bp, \bq)$ is one of the pairs $((p), (q))$, $((m, p), (q))$, $((p), (m, q)$, and similarly for $(\br, \bs)$. In what follows we assume we are not in any of the six cases listed in the lemma.

(1)~$(\bp, \bq) = ((p), (q))$.

(1.1)~$(\br, \bs) = ((r), (s))$.

(1.1.1)~$p = q$. Note that in this case condition~\eqref{cond one} implies $0 < p < m$.

(1.1.1.1)~$r = s$. Note that in this case condition~\eqref{cond two} implies $0 < r < m$. By symmetry we may assume $p \geq r$. In this case, the sequence
\[
0 \to U_{(p)} \to U_{(p + 1)} \oplus U_{(r - 1)} \to U_{(r)} \to 0
\]
of $\Lambda$-modules induces the sequence
\[
0 \to M_{\bp, \bq} \to M_{(p + 1), (p + 1)} \oplus M_{(r - 1, r - 1)} \to M_{\br, \bs} \to 0
\]
of $A$-modules, which implies
\[
\cS_{\bp, \bq} \oplus \cS_{\br, \bs} \subseteq \ov{\cS_{(p + 1), (p + 1)} \oplus \cS_{(r - 1), (r - 1)}}
\]
according to Proposition~\ref{prop generic} (and~\cite{Bongartz1996}*{Lemma~1.1}).

(1.1.1.2)~$r \neq s$. By symmetry we may assume $r > s$. If $p \geq r$, then we have a sequence of inclusions
\[
\cS_{\bp, \bq} \oplus \cS_{\br, \bs} \subseteq \cS_{(p, r), (q, s)} \subseteq \ov{\cS_{(p + 1, r - 1), (q, s)}},
\]
where the latter inclusion follows from the dual of Lemma~\ref{lem def ind}. Assume $r > p$. If $s \geq q$, then
\[
\cS_{\bp, \bq} \oplus \cS_{\br, \bs} \subseteq \cS_{(r, p), (s, q)} \subseteq \ov{\cS_{(r, p), (s + 1, q - 1)}}
\]
by Lemma~\ref{lem def ind}. If $0 < s < q$, then
\[
\cS_{\bp, \bq} \oplus \cS_{\br, \bs} \subseteq \cS_{(r, p), (q, s)} \subseteq \ov{\cS_{(r, p), (q + 1, s - 1)}}
\]
by Lemma~\ref{lem def ind}. Finally, if $s = 0$, then condition~\eqref{cond five} implies $r < m$, hence
\[
\cS_{\bp, \bq} \oplus \cS_{\br, \bs} \subseteq \cS_{(r, p), (q, 0)} \subseteq \ov{\cS_{(r + 1, p - 1), (q, 0)}}
\]
by the dual of Lemma~\ref{lem def ind}.

(1.1.2)~$p \neq q$. By symmetry we may assume $p > q$.

(1.1.2.1)~$r = s$. Note that by condition~\eqref{cond two} $0 < r < m$. If $q \geq r$, then
\[
\cS_{\bp, \bq} \oplus \cS_{\br, \bs} \subseteq \cS_{(p, r), (q, s)} \subseteq \ov{\cS_{(p, r), (q + 1, s - 1)}}
\]
by Lemma~\ref{lem def ind}. Similarly, if $r \geq p$, then
\[
\cS_{\bp, \bq} \oplus \cS_{\br, \bs} \subseteq \cS_{(r, p), (s, q)} \subseteq \ov{\cS_{(r + 1, p - 1), (s, q)}}.
\]
Thus assume $p > r > q$. If $q > 0$, then
\[
\cS_{\bp, \bq} \oplus \cS_{\br, \bs} \subseteq \cS_{(p, r), (s, q)} \subseteq \ov{\cS_{(p, r), (s + 1, q - 1)}}.
\]
Otherwise, $p < m$ by condition~\eqref{cond three} and
\[
\cS_{\bp, \bq} \oplus \cS_{\br, \bs} \subseteq \cS_{(p, r), (s, 0)} \subseteq \ov{\cS_{(p + 1, r - 1), (s, 0)}}.
\]

(1.1.2.2)~$r > s$. By symmetry we may assume $p \geq r$. If $q > r$, then
\[
\cS_{\bp, \bq} \oplus \cS_{\br, \bs} \subseteq \cS_{(p, r), (q, s)} \subseteq \ov{\cO_{M_{(p, r), (q, s)}}}
\]
by Proposition~\ref{prop generic}, and $M_{(p, r), (q, s)} \notin \cS_{\bp, \bq} \oplus \cS_{\br, \bs}$ ($M_{(p, r), (q, s)}$ is indecomposable by Lemma~\ref{lemma ind}). If $p > r \geq q > s$, then
\[
\cS_{\bp, \bq} \oplus \cS_{\br, \bs} \subseteq \cS_{(p, r), (q, s)} \subseteq \ov{\cO_{M_{(p), (s)} \oplus M_{(r), (q)}}}
\]
by Proposition~\ref{prop generic}, and $M_{(p), (s)} \oplus M_{(r), (q)} \notin \cS_{\bp, \bq} \oplus \cS_{\br, \bs}$. Finally assume $p = r$. In this case by symmetry we may assume $q \geq s$. If $s > 0$, then
\[
\cS_{\bp, \bq} \oplus \cS_{\br, \bs} \subseteq \cS_{(p, r), (q, s)} \subseteq \ov{\cS_{(p, r), (q + 1, s - 1)}}
\]
by Lemma~\ref{lem def strata}, while if $s = 0$, then $p = r < m$ by condition~\eqref{cond five} and
\[
\cS_{\bp, \bq} \oplus \cS_{\br, \bs} \subseteq \cS_{(p, r), (q, 0)} \subseteq \ov{\cS_{(p + 1, r - 1), (q, 0)}}
\]
by the dual of Lemma~\ref{lem def strata}. 

(1.1.2.3)~$r < s$. By duality we may assume $p \geq s$. If $q \geq s$, then
\[
\cS_{\bp, \bq} \oplus \cS_{\br, \bs} \subseteq \cS_{(p, r), (q, s)} \subseteq \ov{\cS_{(p, r), (q + 1, s - 1)}}
\]
by Lemma~\ref{lem def strata}. If $s > q \geq r$, then
\[
\cS_{\bp, \bq} \oplus \cS_{\br, \bs} \subseteq \cS_{(p, r), (s, q)} \subseteq \ov{\cO_{M_{(p), (s)} \oplus M_{(r), (q)}}}
\]
by Proposition~\ref{prop generic}, and $M_{(p), (s)} \oplus M_{(r), (q)} \notin \cS_{\bp, \bq} \oplus \cS_{\br, \bs}$. Finally, if $r > q$, then
\[
\cS_{\bp, \bq} \oplus \cS_{\br, \bs} \subseteq \cS_{(p, r), (s, q)} \subseteq \ov{\cO_{M_{(p, r), (s, q)}}}
\]
by Proposition~\ref{prop generic}, and $M_{(p, r), (s, q)} \notin \cS_{\bp, \bq} \oplus \cS_{\br, \bs}$.

(1.2)~$(\br, \bs) = ((m, r), (s))$. Then $m > s > r > 0$ by Corollary~\ref{cor ind irr}.

(1.2.1)~$p = q$. Then $0 < p < m$ by condition~\eqref{cond one}. If $p > s$, then
\[
\cS_{\bp, \bq} \oplus \cS_{\br, \bs} \subseteq \cS_{(m, p, r), (q, s)} \subseteq \ov{\cS_{(m, p, r), (q + 1, s - 1)}}
\]
by Lemma~\ref{lem def ind}. If $s \geq p > r$, then
\[
\cS_{\bp, \bq} \oplus \cS_{\br, \bs} \subseteq \cS_{(m, p, r), (s, q)} \subseteq \ov{\cS_{(m, p, r), (s + 1, q - 1)}}.
\]
by Lemma~\ref{lem def ind} again. Finally, if $r \geq p$, then
\[
\cS_{\bp, \bq} \oplus \cS_{\br, \bs} \subseteq \cS_{(m, r, p), (s, q)} \subseteq \ov{\cS_{(m, r + 1, p - 1), (s, q)}}
\]
by the dual of Lemma~\ref{lem def ind}.

(1.2.2)~$p > q$. Then either $p < m$ or $q > 0$ by condition~\eqref{cond three}. If $q \geq s$, then
\[
\cS_{\bp, \bq} \oplus \cS_{\br, \bs} \subseteq \cS_{(m, p, r), (q, s)} \subseteq \ov{\cS_{(m, p, r), (q + 1, s - 1)}}
\]
by Lemma~\ref{lem def strata}. If $p \geq s > q \geq r$, then
\[
\cS_{\bp, \bq} \oplus \cS_{\br, \bs} \subseteq \cS_{(m, p, r), (s, q)} \subseteq \ov{\cO_{M_{(m, r), (q)} \oplus M_{(p), (s)}}}
\]
by Proposition~\ref{prop generic}, and $M_{(m, r), (q)} \oplus M_{(p), (s)} \notin \cS_{\bp, \bq} \oplus \cS_{\br, \bs}$. If $p \geq s > r > q$, then
\[
\cS_{\bp, \bq} \oplus \cS_{\br, \bs} \subseteq \cS_{(m, p, r), (s, q)} \subseteq \ov{\cO_{M_{(m), (0)} \oplus M_{(p, r), (s, q)}}}
\]
and $M_{(m), (0)} \oplus M_{(p, r), (s, q)} \notin \cS_{\bp, \bq} \oplus \cS_{\br, \bs}$. Next, if $s > p > q \geq r$, then
\[
\cS_{\bp, \bq} \oplus \cS_{\br, \bs} \subseteq \cS_{(m, p, r), (s, q)} \subseteq \ov{\cS_{(m, p + 1, r - 1), (s, q)}}
\]
by the dual of Lemma~\ref{lem def ind}. If $s > p > r > q$, then
\[
\cS_{\bp, \bq} \oplus \cS_{\br, \bs} \subseteq \cS_{(m, p, r), (s, q)} \subseteq \ov{\cO_{M_{(m, p), (s)} \oplus M_{(r), (q)}}}
\]
and $M_{(m, p), (s)} \oplus M_{(r), (q)} \notin \cS_{\bp, \bq} \oplus \cS_{\br, \bs}$. Finally assume that $r \geq p$.
Then
\[
\cS_{\bp, \bq} \oplus \cS_{\br, \bs} \subseteq \cS_{(m, r, p), (s, q)} \subseteq \ov{\cS_{(m, r + 1, p - 1), (s, q)}}
\]
by the dual of Lemma~\ref{lem def strata}.

(1.2.3)~$p < q$. In particular, $q > 0$. If $p \geq s$, then
\[
\cS_{\bp, \bq} \oplus \cS_{\br, \bs} \subseteq \cS_{(m, p, r), (q, s)} \subseteq \ov{\cS_{(m, p + 1, r - 1), (q, s)}}
\]
by the dual of Lemma~\ref{lem def ind}. If $q > s > p \geq r$, then
\[
\cS_{\bp, \bq} \oplus \cS_{\br, \bs} \subseteq \cS_{(m, p, r), (q, s)} \subseteq \ov{\cO_{M_{(p), (s)} \oplus M_{(m, r), (q)}}}
\]
by Proposition~\ref{prop generic}, and $M_{(p), (s)} \oplus M_{(m, r), (q)} \notin \cS_{\bp, \bq} \oplus \cS_{\br, \bs}$. If $q > s > r > p$, then
\[
\cS_{\bp, \bq} \oplus \cS_{\br, \bs} \subseteq \cS_{(m, r, p), (q, s)} \subseteq \ov{\cO_{M_{(r), (s)} \oplus M_{(m, p), (q)}}}
\]
and $M_{(r), (s)} \oplus M_{(m, p), (q)} \notin \cS_{\bp, \bq} \oplus \cS_{\br, \bs}$. If $s \geq q > p \geq r$, then
\[
\cS_{\bp, \bq} \oplus \cS_{\br, \bs} \subseteq \cS_{(m, p, r), (s, q)} \subseteq \ov{\cS_{(m, p, r), (s + 1, q - 1)}}
\]
by Lemma~\ref{lem def strata}. If $s \geq q \geq r > p$, then
\[
\cS_{\bp, \bq} \oplus \cS_{\br, \bs} \subseteq \cS_{(m, r, p), (s, q)} \subseteq \ov{\cO_{M_{(r), (q)} \oplus M_{(m, p), (s)}}}
\]
and $M_{(r), (q)} \oplus M_{(m, p), (s)} \notin \cS_{\bp, \bq} \oplus \cS_{\br, \bs}$. Finally, if $r > q$, then
\[
\cS_{\bp, \bq} \oplus \cS_{\br, \bs} \subseteq \cS_{(m, r, p), (s, q)} \subseteq \ov{\cS_{(m, r, p), (s + 1, q - 1)}}
\]
by Lemma~\ref{lem def ind}.

(1.3)~$(\br, \bs) = ((r), (m, s))$. This is dual to~(1.2).

(2)~$(\bp, \bq) = ((m, p), (q))$. Then $m > q > p > 0$ by Corollary~\ref{cor ind irr}.

(2.1)~$(\br, \bs) = ((r), (s))$. This is symmetric to~(1.2).

(2.2)~$(\br, \bs) = ((m, r), (s))$. Then $m > s > r > 0$ by Corollary~\ref{cor ind irr}. By symmetry we may assume $q \geq s$. If in addition $p \geq s$, then
\[
\cS_{\bp, \bq} \oplus \cS_{\br, \bs} \subseteq \cS_{(m, m, p, r), (q, s)} \subseteq \ov{\cO_{M_{(m), (0)} \oplus M_{(m, p, r), (q, s)}}}
\]
by Proposition~\ref{prop generic}, and $M_{(m), (0)} \oplus M_{(m, p, r), (q, s)} \notin \cS_{\bp, \bq} \oplus \cS_{\br, \bs}$. Thus we assume $s > p$, and put $p' := \max \{ p, r \}$ and $r' := \min \{ p, r \}$. Then
\[
\cS_{\bp, \bq} \oplus \cS_{\br, \bs} \subseteq \cS_{(m, m, p', r'), (q, s)} \subseteq \ov{\cS_{(m, m, p', r'), (q + 1, s - 1)}}
\]
by Lemma~\ref{lem def strata}.

(2.3)~$(\br, \bs) = ((r), (m, s))$. Then $m > r > s > 0$ by Corollary~\ref{cor ind irr}. In this case Proposition~\ref{prop generic} implies that
\[
\cS_{\bp, \bq} \oplus \cS_{\br, \bs} \subseteq \cS_{\bp \cup \br, \bq \cup \bs} \subseteq \ov{\cO_{M_{(m), (m)} \oplus M}},
\]
for some $A$-module $M$, and $M_{(m), (m)} \oplus M \notin \cS_{\bp, \bq} \oplus \cS_{\br, \bs}$.

(3)~$(\bp, \bq) = ((p), (m, q))$. This is dual to~(2).
\end{proof}

\subsection{Direct sum decompositions of irreducible components} \label{subsect KS}
Before we continue our proof we need to present the Krull--Remak--Schmidt theory for irreducible components as developed in \cite{CBSch} (see also~\cite{delaPena}). First, if $\cY$ is an irreducible component of $\cX_m (d_0, d_1)$, for some $d_0, d_1 \geq 0$, then there exist indecomposable irreducible components $\cY_1$, \ldots, $\cY_n$ such that $\cY = \ov{\cY_1 \oplus \cdots \oplus \cY_n}$. On the other hand, if $\cY_1$, \ldots, $\cY_n$ are indecomposable irreducible components, then $\ov{\cY_1 \oplus \cdots \oplus \cY_n}$ is an irreducible component if and only if $\ext (\cY_i, \cY_j) = 0$, for all $i \neq j$. Here, for two subsets $\cY'$ and $\cY''$ of $\cX_m (d_0', d_1')$ and $\cX_m (d_0'', d_1'')$, respectively, we put
\[
\ext (\cY', \cY'') := \min \{ \dim_K \Ext_A^1 (M', M'') \mid \text{$M' \in \cY'$ and $M'' \in \cY''$} \}.
\]
Taking the above into account, the following is a reformulation of Lemma~\ref{lem direct sums}.

\begin{Cor} \label{cor ext}
Let $(\bp, \bq)$ and $(\br, \bs)$ be pairs of partitions such that $\ov{\cS_{\bp, \bq}}$ and $\ov{\cS_{\br, \bs}}$ are indecomposable irreducible components. Then
\renewcommand{\theequation}{$*$}
\begin{equation} \label{eq non zero ext}
\ext (\cS_{\bp, \bq}, \cS_{\br, \bs}) \neq 0 \qquad \text{or} \qquad  \ext (\cS_{\br, \bs}, \cS_{\bp, \bq}) \neq 0
\end{equation}
unless one of the following conditions is satisfied:
\begin{enumerate}

\item \label{cond one prim}
$(\bp, \bq) = ((m), (m))$, or

\item
$(\br, \bs) = ((m), (m))$, or

\item
$(\bp, \bq) = ((m), (0))$ and $(\br, \bs)$ is of mono type, or

\item
$(\bp, \bq) = ((0), (m))$ and $(\br, \bs)$ is of epi type, or

\item
$(\br, \bs) = ((m), (0))$ and $(\bp, \bq)$ is of mono type, or

\item \label{cond six prim}
$(\br, \bs) = ((0), (m))$ and $(\bp, \bq)$ is of epi type.

\end{enumerate}
\end{Cor}

It is easy to observe that in all the cases listed above $(\bp, \bq) \cup (\br, \bs)$ consists of maximal partitions, thus $\cS_{(\bp, \bq) \cup (\br, \bs)}$ is an irreducible component of the corresponding variety. Moreover, $(\bp, \bq)$ and $(\br, \bs)$ form the canonical decomposition of $(\bp, \bq) \cup (\br, \bs)$. Consequently, conditions \eqref{cond one prim}--\eqref{cond six prim} are in fact equivalent to condition~\eqref{eq non zero ext}.

\subsection{Proof of Theorem~\ref{thm:intro4}}
Let $d_0, d_1 \geq 0$ and $\bp_0$ and $\bq_0$ be the maximal partitions in $\cP_m (d_0)$ and $\cP_m (d_1)$, respectively. Let $\cY$ be an irreducible component of $\cX_m (d_0, d_1)$ and $\cS_{\bp_1, \bq_1}$, \ldots, $\cS_{\bp_k, \bq_k}$ be indecomposable irreducible components such that
\[
\cY = \ov{\cS_{\bp_1, \bq_1} \oplus \cdots \oplus \cS_{\bp_k, \bq_k}}.
\]
Then $\ext (\cS_{\bp_i, \bq_i}, \cS_{\bp_j, \bq_j}) = 0$ for all $i \neq j$ by Subsection~\ref{subsect KS}.

Assume first there exists $i$ such that $(\bp_i, \bq_i)$ is of mono type, but $(\bp_i, \bq_i) \neq ((m), 0)$ and $(\bp_i, \bq_i)$ is not of epi type. Then Corollary~\ref{cor ext} implies that either $(\bp_j, \bq_j) = ((m), (m))$ or $(\bp_j, \bq_j) = ((m), (0))$, for each $j \neq i$. Consequently,
\[
(\bp_1, \bq_1) \cup \cdots \cup (\bp_k, \bq_k) = (\bp_0, \bq_0),
\]
thus
\[
\cY = \ov{\cS_{\bp_1, \bq_1} \oplus \cdots \oplus \cS_{\bp_k, \bq_k}} \subseteq \ov{\cS_{\bp_0, \bq_0}}.
\]
Since $\cY$ is an irreducible component and $\cY \subseteq \ov{\cS_{\bp_0, \bq_0}}$, $\cY = \ov{\cS_{\bp_0, \bq_0}}$. We proceed similarly, if there exists $i$ such that $(\bp_i, \bq_i)$ is of epi type, but $(\bp_i, \bq_i) \neq ((0), (m))$ and $(\bp_i, \bq_i)$ is not of mono type

Now assume that for each $i$ either $(\bp_i, \bq_i)$ is both of mono and epi type, or $(\bp_i, \bq_i) = ((m), (0))$, or $(\bp_i, \bq_i) = ((0), (m))$. If $(\bp_i, \bq_i)$ if both of mono and epi type, then Corollary~\ref{cor ind irr} implies that $\bp_i = (n) = \bq_i$ for some $1 \leq n \leq m$. Let $l$ be the number of $i$'s such that $(\bp_i, \bq_i)$ is both of mono and epi type, but $((\bp_i), (\bq_i)) \neq ((m), (m))$. Then Corollary~\ref{cor ext} implies $l = 1$, hence we again get
\[
(\bp_1, \bq_1) \cup \cdots \cup (\bp_k, \bq_k) = (\bp_0, \bq_0),
\]
i.e.\ $\cY = \ov{\cS_{\bp_0, \bq_0}}$, which finishes the proof. \qed

\subsection{Proof of Corollary~\ref{cor main}}
A geometric version of Morita equivalence due to Bongartz~\cite{Bongartz1991} implies that the algebra $A (m, 1)$ is geometrically irreducible if and only if $\cX_m (d_0, d_1)$ is irreducible for all $d_0, d_1 \geq 0$. Consequently, Corollary~\ref{cor main} is equivalent to Theorem~\ref{thm:intro4}. \qed

\bibsection


\begin{biblist}

\bib{Birkhoff}{article}{
   author={Birkhoff, G.},
   title={Subgroups of abelian groups},
   journal={Proc. Lond. Math. Soc. (II)},
   volume={38},
   date={1935},
   pages={385--401},
}

\bib{BobinskiSchroer2017a}{article}{
   author={Bobi\'nski, G.},
   author={Schr\"oer, J.},
   title={Algebras with irreducible module varieties I},
   eprint={arXiv:1709.05841},
}

\bib{BobinskiSchroer2017b}{article}{
   author={Bobi\'nski, G.},
   author={Schr\"oer, J.},
   title={Algebras with irreducible module varieties II: Two vertex case},
   eprint={arXiv:1801.03677},
}

\bib{Bongartz1991}{article}{
   author={Bongartz, K.},
   title={A geometric version of the Morita equivalence},
   journal={J. Algebra},
   volume={139},
   date={1991},
   pages={159--171},
}

\bib{Bongartz1996}{article}{
   author={Bongartz, K.},
   title={On degenerations and extensions of finite-dimensional modules},
   journal={Adv. Math.},
   volume={121},
   date={1996},
   pages={245--287},
}

\bib{CKW}{article}{
  author={Chindris, C.},
  author={Kinser, R.},
  author={Weyman, J.},
  title={Module varieties and representation type of finite-dimensional     algebras},
  journal={Int. Math. Res. Not. IMRN},
  volume={2015},
  date={2015},
  pages={631--650},
}

\bib{CBSch}{article}{
   author={Crawley-Boevey, W.},
   author={Schr\"oer, J.},
   title={Irreducible components of varieties of modules},
   journal={J. Reine Angew. Math.},
   volume={553},
   date={2002},
   pages={201--220},
}

\bib{delaPena}{article}{
   author={de la Pe\~na, J. A.},
   title={On the dimension of the module-varieties of tame and wild algebras},
   journal={Comm. Algebra},
   volume={19},
   date={1991},
   pages={1795--1807},
}

\bib{GeissLeclercSchroer}{article}{
   author={Geiss, Ch.},
   author={Leclerc, B.},
   author={Schr\"{o}er, J.},
   title={Quivers with relations for symmetrizable Cartan matrices I: Foundations},
   journal={Invent. Math.},
   volume={209},
   date={2017},
   pages={61--158},
}

\bib{Gerstenhaber}{article}{
   author={Gerstenhaber, M.},
   title={On dominance and varieties of commuting matrices},
   journal={Ann. of Math. (2)},
   volume={73},
   date={1961},
   pages={324--348},
}

\bib{KLM}{article}{
   author={Kussin, D.},
   author={Lenzing, H.},
   author={Meltzer, H.},
   title={Nilpotent operators and weighted projective lines},
   journal={J. Reine Angew. Math.},
   volume={685},
   date={2013},
   pages={33--71},
}

\bib{LuoZhang}{article}{
   author={Luo, X.-H.},
   author={Zhang, P.},
   title={Monic representations and Gorenstein-projective modules},
   journal={Pacific J. Math.},
   volume={264},
   date={2013},
   pages={163--194},
}

\bib{Richman}{collection.article}{
   author={Richman, F.},
   author={Walker, E. A.},
   title={Subgroups of $p^5$-bounded groups},
   booktitle={Abelian Groups and Modules},
   series={Trends Math.},
   publisher={Birkh\"{a}user, Basel},
   date={1999},
   pages={55--73},
}

\bib{Riedtmann}{article}{
   author={Riedtmann, C.},
   title={Degenerations for representations of quivers with relations},
   journal={Ann. Sci. \'Ecole Norm. Sup. (4)},
   volume={19},
   date={1986},
   pages={275--301},
}

\bib{Ringel}{article}{
  author={Ringel, C. M.},
  author={Schmidmeier, M.},
  title={Invariant subspaces of nilpotent linear operators. I},
  journal={J. Reine Angew. Math.},
  volume={614},
  date={2008},
  pages={1--52},
}

\end{biblist}
\end{document}